\documentclass[article]{amsart}
\usepackage[utf8]{inputenc}    
 
\usepackage[french,english]{babel}     
 
\usepackage[cyr]{aeguill}      
\usepackage{amsmath, amsthm, amsfonts,stmaryrd }
\usepackage{mathrsfs}                      
\usepackage{mathtools}                     
\usepackage{booktabs}
\usepackage{todonotes}
\usepackage[all]{xy}
\usepackage[utf8]{inputenc}
\usepackage{mathtools}                     
\usepackage{todonotes}
\usepackage{ marvosym }
\usepackage{bbm}                         
\usepackage{framed}
\usepackage{graphicx}   

\usepackage{hyperref} 
\hypersetup{
    colorlinks=true,
    linkcolor=blue,  
    urlcolor=cyan,
    pdftitle={RiemannianSplines},
    pdfauthor={R. TAHRAOUI, F.X. VIALARD},
} 
\usepackage{theoremref}

\numberwithin{equation}{section}
\theoremstyle{plain}
\newtheorem{remark}{Remark}
\newtheorem{theorem}{Theorem}
\newtheorem{lemma}[theorem]{Lemma}

\newtheorem{proposition}[theorem]{Proposition}

\theoremstyle{definition}

\newcommand{\ms}[1]{\mathbb{#1}}
\newcommand{\mc}[1]{\mathcal{#1}}

\def\om{\omega}

\def\Om{\Omega}
 
\def\o{\circ} 
\def\i{^{-1}}

\def\R{{\mathbb R}}

\def\ad{\operatorname{ad}}

\let\on=\operatorname

\newcommand{\ud}{\,\mathrm{d}}

\let\mc=\mathcal
\let\mf=\mathfrak
\def\Inc=\operatorname{Inc}

\newcommand{\eqdef}{\ensuremath{\stackrel{\mbox{\upshape\tiny def.}}{=}}}

\begin{document}
\title[Variational higher-order interpolation on the group of diffeomorphisms]{Variational second-order interpolation on the group of diffeomorphisms with a right-invariant metric}
\author{Fran\c{c}ois-Xavier Vialard}
\maketitle
\begin{abstract}
In this note, we propose a variational framework in which the minimization of the acceleration on the group of diffeomorphisms endowed with a right-invariant metric is well-posed. It relies on constraining the acceleration to belong to a Sobolev space of higher-order than the order of the metric in order to gain compactness. It provides the theoretical guarantee of existence of minimizers which is compulsory for numerical simulations.
\end{abstract}

\section{Introduction}

The question of interpolating a time-sequence of shapes with a curve and representing a shape evolution with few parameters have been addressed in the literature related to shape analysis and medical imaging since the last ten years. Several methods have been proposed and studied and they essentially rely on extension of standard tools available in Euclidean geometry to shape spaces. In this direction, we mention geodesic regression, cubic regression, kernel methods... The generalization of these tools to infinite dimensional setting are sometimes complicated by the fact that the shape space is not a flat space, nor a finite dimensional space. However, the shape space is usually endowed with a Riemannian structure and most often in infinite dimensions. The generalizations of this Euclidean tools are often introduced by variational formulations, the simplest example being the case of shortest path between two shapes, i.e. geodesics on the space of shapes. Even in that particular example, finding a variational setting in which the object of interest is well defined is of interest, since the existence of an extremum is not guaranteed in general and is complicated by the infinite dimensional setting. For example, in the case of group of diffeomorphisms, this question is addressed in \cite{CompletenessDiffeomorphismGroup}, in which the authors prove that the group of diffeomorphisms of the Euclidean space endowed with a right-invariant Sobolev metric of high enough order is complete in the sense of the Hopf-Rinow theorem. The case of the group of diffeomorphisms with right-invariant metric is relevant for applications in medical imaging and in particular for the problem of diffeomorphic image matching \cite{laurentbook,0855.57035,BegIJCV}. It is also natural to study and develop higher-order interpolations in the space of shapes, which has been actively developed in finite dimensions \cite{birkhoff65,VariationalStudySplines,Noakes1,splinesCk,Crouch,splinesanalyse,Koiso}. It was also extensively used and numerically developed in image processing and computer vision \cite{MumfordElastica,Masnou1,PamiSplines,Samir,CAD,Chan02eulerselastica}. In the past few years, these higher-order models have been introduced in biomedical imaging for interpolation of a time sequence of shapes. They have been proposed in \cite{TrVi2010} for a diffeomorphic group action on a finite dimensional manifold and further developed for general invariant higher-order lagrangians in \cite{HOSplines1,HOSplines2} on a group. A numerical implementation together with a generalized model have been proposed in \cite{SinghVN15} in the context of medical imaging applications. However, in all these articles, the question of existence of an extremum is not treated. An attempt is given in \cite{Vialard2016} where the exact relaxation of the problem is shown on the group of diffeomorphisms of the interval $[0,1]$. The main result of \cite{Vialard2016} consists in providing the existence of a minimizer in a larger space where the relaxation is defined. Although it does not completely solve the problem, it shows that existence of cubic splines for a group of diffeomorphisms with a right-invariant metric is non trivial. Let us discuss where the difficulty comes in a Riemannian setting. Riemannian splines are minimizers of 
\begin{equation}\label{SplinesFunctional}
\mathcal{J}(x) = \int_0^1 g\left( \frac{D}{Dt} \dot{x},\frac{D}{Dt}\dot{x}  \right) \ud t\,,
\end{equation}
where $(M,g)$ is a Riemannian manifold, $\frac{D}{Dt}$ is its associated covariant derivative and $x$ is a sufficiently smooth curve from $[0,1]$ in $M$ satisfying first order boundary conditions, i.e. $x(0),\dot{x}(0)$ and $x(1),\dot{x}(1)$ are fixed.
The term $$\frac{D}{Dt} \dot{x} = \ddot{x}  + \Gamma(x)(\dot{x},\dot{x})$$ (written in coordinates, with $\Gamma$ the Christoffel symbols) contains nonlinearities which contribute in the variational problem \eqref{SplinesFunctional} by possibly generating high-frequency oscillations in the space variable. 
\par
Although this notion of Riemannian cubics could not be well defined in general, it is possible to slightly modify it to make it well-posed. A modification of this type has recently been proposed in \cite{WirthSplines} in their framework.
In this paper, we propose a simple variational setting which makes the second-order variational problem well-posed at the expense of increasing the regularity of the group of diffeomorphisms on which the second-order interpolation is feasible. For practical applications, this gain of smoothness, or loss of controllability of the diffeomorphism does not matter so much since smoothness is preferred in medical image registration. However, the theoretical existence is guaranteed. The main result of the paper is the following
\begin{theorem}[Main result]
Let $\Om$ be a bounded domain in $\R^d$ and $s'\geq s+1$.
There exists a minimizer to the functional  
\begin{equation}\label{SplinesFunctional}
\mathcal{J}(x) = \int_0^1 \left\| \frac{D}{Dt} \dot{x} \o x\i \right\|^2_{H^{s'}}  \ud t\,,
\end{equation}
on the loop space $\Om_{0,1}(\mc G_{H^{s'}(\Om,\R^d)})$ where $ \frac{D}{Dt}$ is the covariant derivative associated with the right-invariant $H^s$ metric.
\end{theorem}
In other words, the acceleration is measured in a stronger space than the ambient space so that it will prevent from creating oscillations.

\section{Background on right-invariant metrics on diffeomorphisms group} \label{Background}

\par \textbf{Sobolev right-invariant metrics on the group of diffeomorphisms.} In \cite{CompletenessDiffeomorphismGroup}, the authors proved the following theorem
Let $M$ be either $\R^d$ or a compact manifold without boundary of dimension $d$. We define hereafter a group of diffeomorphisms of $M$ which is a complete metric space.
Consider a space $V$ a Hilbert space of vector fields on $M$ (rapidly decreasing at infinity in the unbounded case), left invariant by their flows, such that the inclusion map $V \hookrightarrow W^{1,\infty}(M,\R^d)$ is continuous. This hypothesis implies that the flow of a time dependent vector field in $L^2([0,1],V)$ is well defined, see \cite[Appendix C]{laurentbook}. Then, the set of flows at time $1$ defines a group of diffeomorphisms denoted by $\mathcal{G}_V$. Denoting 
\begin{equation}
\on{Fl}_1(\xi) = \varphi(1) 
\end{equation}
where $\varphi$ solves the flow equation
\begin{align} \label{EqFlow}
&\partial_t \varphi(t,x) = \xi(t,\varphi(t,x))\\
&\varphi(0,x) = x\, \text{ } \forall x \in D\,,
\end{align}
we define
\begin{equation}\label{EqDefGV}
\mathcal{G}_V \eqdef \{ \varphi(1) \,: \, \exists \, \xi \in L^2([0,1],V) \text{ s.t. } \on{Fl}_1(\xi) \}\,,
\end{equation}
which has been introduced by Trouvé in \cite{0855.57035}. On this group, Trouvé  defines a metric 
\begin{equation}
\on{dist}(\psi_1,\psi_0)^2 = \inf_{} \left\{\int_0^1 \| \xi \|_V^2 \ud t \,: \, \xi \in L^2([0,1],V) \text{ s.t. } \psi_1 = \on{Fl}_1(\xi) \circ \psi_0 \,\right\} 
\end{equation}
under which he proves that $\mathcal{G}_V$ is complete. In full generality, that is for a general space of vector fields $V$, very few properties are known on this group. For instance, it is a priori not a topological group, or more precisely, there is no known topological structure making it a topological group (the inversion need not be continuous). Moreover, there does not need to be a differentiable structure on this group.
However, for certain choices of spaces $V$, such structures are available and therefore more properties can be derived in this situation. Indeed, consider the group $\mc D^s(M)$, with $s > d/2+1$, which consists of all $C^1$-diffeomorphisms of Sobolev regularity $H^s$. It is known since the work of Ebin and Marsden \cite{Ebin1970} that $\mc D^s(M)$ is a smooth Hilbert manifold and a topological group. It only remains to prove that $\mathcal{G}_{H^s} = \mc D^s(M)_0$ (the connected component of identity) which is done in \cite[Section 8]{CompletenessDiffeomorphismGroup} and its main result is

\begin{theorem}\label{Thmdiff_hopf_rinow}
Let $M$ be $\R^d$ or a closed manifold and $s > d/2+1$. If $G^s$ is a smooth, right-invariant Sobolev-metric of order $s$ on $\mc D^s(M)$, then
\begin{enumerate}
\item
$(\mc D^s(M), G^s)$ is geodesically complete;
\item
$(\mc D^s(M)_0, \on{dist}^s)$ is a complete metric space;
\item
Any two elements of $\mc D^s(M)_0$ can be joined by a minimizing geodesic.
\end{enumerate}
The statements also hold for the subgroups $\mc D^s_\mu(M)$ and $\mc D^s_\om(M)$ of diffeomorphisms preserving a volume form $\mu$ or a symplectic structure $\om$.
\end{theorem}

The crucial ingredient in the proof is showing that the flow map
\begin{equation}
\label{eq:intro_flow}
\on{Fl}_t : L^1(I, \mf X^s(M)) \to \mc D^s(M)
\end{equation}
exists and is continuous.
\par
In \cite{TrVi2010}, we introduced the use of cubic splines \textit{in the space of shapes} to interpolate a sequence of shapes that are time dependent.
Riemannian cubics (also called Riemannian splines) and probably more famous, its constrained alternative called Elastica belong to a class of problems that have been studied since the work of Euler (see the discussion in \cite{MumfordElastica}). Let us present the variational problem in a Riemannian setting. Riemannian splines are minimizers of 
\begin{equation}\label{SplinesFunctional}
\mathcal{J}(x) = \int_0^1 g\left( \frac{D}{Dt} \dot{x},\frac{D}{Dt}\dot{x}  \right) \ud t\,,
\end{equation}
where $(M,g)$ is a Riemannian manifold, $\frac{D}{Dt}$ is its associated covariant derivative and $x$ is a sufficiently smooth curve from $[0,1]$ in $M$ satisfying first order boundary conditions, i.e. $x(0),\dot{x}(0)$ and $x(1),\dot{x}(1)$ are fixed.
The case of Elastica consists in restricting the previous optimization problem to the set of curves that are parametrized by unit speed (when the problem is feasible), namely $g(\dot{x},\dot{x})=1$ for all time. To the best of our knowledge, the only paper that deals with analytical questions is \cite{splinesanalyse} where the authors show in particular the existence of minimizers of a second-order functional on the space curves on a complete finite dimensional Riemannian manifold.


In \cite{HOSplines1}, higher-order models are proposed on groups of diffeomorphisms but for the standard Riemannian cubics functional, no analytical study was provided. Indeed, in the case of a Lie group $G$ (and $\mathfrak{g}$ its Lie algebra) with a right-invariant metric ($\| \cdot \|_\mathfrak{g}$ denoting the norm on the Lie algebra), the covariant derivative can be written as follows:
Let $V(t) \in T_{g(t)}G$ be a vector field along a curve $g(t) \in G$
 \begin{equation}\label{Covariant_derivative_vf_prop}
      \frac{D}{Dt} V =\Big( \dot{\nu} + \frac{1}{2} \operatorname{ad}^\dagger_\xi \nu + \frac{1}{2} \operatorname{ad}^\dagger_\nu \xi - \frac{1}{2} [\xi, \nu]\Big) _G(g).
    \end{equation}
    where $\operatorname{ad}^\dagger$ is the metric adjoint defined by
     \begin{equation}\label{ad-dagger-def}
 \operatorname{ad}^\dagger_{\nu}{\kappa} := 
 (\operatorname{ad}^*_\nu (\kappa^\flat))^\sharp
 \end{equation}
 for any $\nu, \kappa\in \mathfrak{g}$ and $\flat$ and $\sharp$ are the musical operator for the 
cometric and metric operator. Therefore, the reduced lagrangian for \eqref{SplinesFunctional} is 
\begin{equation}\label{ReducedSplinesFunctional}
\mathcal{J}(x) = \int_0^1 \| \dot{\xi} + \ad^\dagger_\xi \xi\|^2_\mathfrak{g} \ud t\,.
\end{equation}
 where $\operatorname{ad}^\dagger$ is the metric adjoint, i.e.,
 it is written as
 \begin{equation}\label{ad-dagger-def}
 \operatorname{ad}^\dagger_{\nu}{\kappa} := 
 (\operatorname{ad}^*_\nu (\kappa^\flat))^\sharp
 \end{equation}
 for any $\nu, \kappa\in \mathfrak{g}$. 
 We can also formulate the variational problem on the dual of the Lie algebra $\mathfrak{g}^*$ by
\begin{equation}\label{ReducedSplinesFunctionalDualLieAlgebra}
\mathcal{J}(x) = \int_0^1 \| a(t) \|^2_{\mathfrak{g}^*} \ud t\,,
\end{equation}
under the constraint 
\begin{equation}
\dot{m} + \ad^*_{\xi} m = a\,.
\end{equation}

In infinite dimensions, there is a clear obstacle to use reduction since the operator $ \ad^\dagger$ is unbounded on the tangent space at identity due to a loss of derivative. 
However, using the smooth Riemannian structure on $\mc D^s$, functional \eqref{SplinesFunctional} is well defined.


The following proposition of \cite{splinesanalyse} is valid in infinite dimensions:
\begin{proposition}
Let $(M,g)$ be an infinite dimensional Riemannian manifold and $$\Omega_{0,1}(M) := \{  x \in H^2([0,1],M) \, | \, x(i) = x_i \,,\, \dot{x}(i) = v_i \text{ for } i=0,1 \}$$ be the space of paths with the first order boundary constraints for given $(x_0,v_0) \in TM$ and $(x_1,v_1) \in TM$. 
\noindent
The functional \eqref{SplinesFunctional} is smooth on $\Omega_{0,1}(M)$ and 
\begin{equation}
\mathcal{J}'(x)(v) = \int_0^1 g(\frac{D^2}{Dt^2}\dot{v},\frac{D}{Dt}\dot{x}) - g(R(\dot{x},\frac{D}{Dt}\dot{x}),v) \, \ud t\,.
\end{equation}
A critical point of $\mathcal{J}$ is a smooth curve that satisfies the Riemannian cubic equation
\begin{equation}
\frac{D^3}{Dt^3}\dot{x} - R(\dot{x},\frac{D}{Dt}\dot{x})\dot{x} = 0\,.
\end{equation}
\end{proposition}

The critical points of $\mathcal{J}$ are sometimes called Riemannian cubics or cubic polynomials. The existence of minimizers does not follow from the corresponding proof in \cite{splinesanalyse} since it strongly relies on the finite dimension hypothesis to have compactness properties.

\section{The main result}
We formulate the main result on the flat torus but it can be generalized to bounded domains in $\R^d$ in a straightforward way.
\begin{theorem}\label{ThmSplines}
Let $\ms T^d$ be the $d$ dimensional flat torus and $s'\geq s+1$.
There exists a minimizer to the functional  
\begin{equation}\label{SplinesFunctional}
\mathcal{J}(x) = \int_0^1 \left\| \frac{D}{Dt} \dot{\varphi} \o \varphi \i \right\|^2_{H^{s'}}  \ud t\,,
\end{equation}
on the loop space $\Om_{0,1}(\mc G_{H^{s'}(\ms T^d,\R^d)})$ where $ \frac{D}{Dt}$ is the covariant derivative associated with the right-invariant $H^s$ metric.
\end{theorem}

Before proving the theorem, we prove the following lemmas:
\begin{lemma}\label{ThmExistenceLagrangian}
Let $\alpha \in L^2([0,1],H^s)$, then there exists a unique solution defined on $[0,1]$ to the system
\begin{subequations} \label{ReducedForcedEvolution}
\begin{align}
&\dot{\varphi} = v\\ 
&\dot{v} = -\Gamma(\varphi)(v,v) + \alpha \circ \varphi \,, \label{Acceleration}
\end{align}
\end{subequations}
for given initial conditions $\varphi(0) = \varphi_0 \in D^{s+1}$ and $v(0) = v_0 \in H^{s+1}$.
\end{lemma}

\begin{proof}
Solutions exist for short time since the system is Lipschitz on $\mc D^s \times H^s$ and the existence theorem for Caratheodory equation gives the result.
\\
Existence for all time is not guaranteed a priori since the second equation of is quadratic in $v$. Let us denote $u = v \circ \varphi^{-1}$ and $f(t):= \frac 12 g(\varphi)(v,v) = \frac 12 \| u \|_{H^s}^2$. Deriving it in time gives $f'(t) = \langle \alpha(t), u(t)\rangle_{H^s}$, so that by the Cauchy-Schwarz inequality, we get
\begin{align}\label{GronwallEnergy}
& f(t)  \leq \int_0^t \sqrt{f(z)} \| \alpha \|_{H^s} \ud z \leq  \| \alpha \|_{L^2([0,1],H^s)}  \sqrt{\int_0^t f(z) \ud z}\\
& f(t) \leq  \| \alpha \|_{L^2([0,1],H^s)} \left(1 +  \int_0^t f(z) \ud z \right) \,.
\end{align}
Using Gronwall's lemma, it implies that $f(t)$ is bounded for on $[0,T]$ where $T$ is the supremum (possible blow-up) time of definition. Therefore, $r := \int_0^T \| u \|^2_{H^s} \, \ud t = \int_0^T f(t) \ud t < \infty$ which means that $u \in L^2([0,T],H^s)$. As a consequence, for all time $t \in [0,T[$, $\varphi(t) \in B(x_0,r)$, which is the ball of radius $r$ for the geodesic distance on $D^s$. In addition, $\lim_{t\to T} \varphi(t)$ is well defined since $D^s$ is metrically complete. 
Remark that $\Gamma(\varphi(t))$ is bounded (uniformly in time) as an operator on $H^s \times H^s$ since $\Gamma$ is smooth on $\mc D^s$ and thus continuous on the path $\varphi(t)$. In particular, the right-hand side of the equation \eqref{Acceleration} belongs to $L^1([0,T],H^s)$ so that $v(T):= \int_0^T - \Gamma(\varphi(t))(v(t),v(t)) + \alpha(t)  \o \varphi(t) \ud t$. Using short time existence on $(\varphi(T),v(T))$, the solution can be extended for short time from time $T$ so that in fact $T=1$.
 \end{proof}

In the following lemma, we study the solutions of the system \eqref{ReducedForcedEvolution} but written on the dual of the tangent space at identity. As mentioned in Section \ref{Background}, one can rewrite the minimization as in Equation \eqref{ReducedSplinesFunctionalDualLieAlgebra}, however the "dual" acceleration is measured using with the corresponding dual norm. In our case, the dual norm (w.r.t. to $H^s$) associated with $H^{s+1} \subset H^s$ is $(H^{s-1})^* \subset (H^s)^*$ as can be seen by a direct computation in Fourier spaces.

\begin{lemma}\label{ThmExistenceEulerian}
Let $a \in L^2([0,1],(H^{s-2})^*)$ then the following integral equation
\begin{equation}\label{IntegralFormulation}
m(t) = Ad_{g(t)^{-1}}^*(m(0)) + \int_0^t Ad_{g_{t,s}}^*(a(s)) ds \,,
\end{equation}
with initial condition $m(0) \in (H^{s-2})^*$ has a unique solution in  $C^0([0,T],(H^{s})^*)$.
If $a \in L^2([0,1],(H^{s-1})^*)$, then there exists a solution to the integral equation with initial condition in $m(0) \in (H^{s-1})^*$.
\end{lemma}

\begin{proof}
The proof of the existence for $a \in L^2([0,1],(H^{s-2})^*)$ follows a standard fixed point method. Let $\Psi : C^0([0,T],(H^{s})^*) \to C^0([0,T],(H^{s})^*)$ defined by Formula \eqref{IntegralFormulation} namely
\begin{equation} \label{IntegralFormulaBis}
\Psi(m)(t) = Ad_{g(t)^{-1}}^*(m(0)) + \int_0^t Ad_{g_{t,s}}^*(a(s)) ds\,,
\end{equation}
where $g_{t,s}$ is the flow of diffeomorphims generated by the vector field associated with $m$.
Remark first that $\Psi(m)$ lies in $C^0([0,T],(H^{s})^*)$ which is well-defined since the integrand is integrable.
Now we claim that the map is a contraction on $(C^0([0,T],(H^{s})^*),\|\cdot\|_{\infty})$ for $T$ small enough,
\begin{multline}
\|\Psi(m_1) - \Psi(m_2)\|_{\infty} \leq \sup_{t \in [0,T]} \| Ad_{g_1(t)^{-1}}^*(m(0)) - Ad_{g_2(t)^{-1}}^*(m(0)) \|_{(H^{s})^*} + \\
\int_0^T \|Ad_{{g_1}_{t,s}}^*a(s)-Ad_{{g_2}_{t,s}}^*a(s)\|_{(H^{s})^*} \, ds \,.
\end{multline}
We need to estimate for $\alpha \in (H^{s-2})^*$ and $w \in H^{s}$. There exists a constant $M_0 > 0$ such that
\begin{align*}
\langle Ad_{g_1}^*(\alpha) - Ad_{g_2}^*(\alpha),w \rangle_{L^2} &\leq \| \alpha \|_{(H^{s-2})^*}   \| Ad_{g_1}(w) - Ad_{g_2}(w) \|_{H^{s-2}}\\
&M_0 \| \alpha \|_{(H^{s-2})^*}   \| g_1 - g_2 \|_{H^{s-2}} \| w\|_{H^{s}}
\end{align*}
Moreover, there exists a constant $M_1 > 0$ s.t. $$\sup_{s,t \in [0,T]} \| {g_1}_{t,s} - {g_2}_{t,s} \|_{H^{s-2}} \leq M_0 \| m_1 - m_2\|_{L^2([0,T],(H^s)^*)}\,. $$
Therefore there exists $M_1$ s.t.
\begin{equation} \label{LipProp}
\|\Psi(m_1) - \Psi(m_2)(t)\|_\infty \leq M_1 \sqrt{t} \| m_1 - m_2\|_{\infty} ( \| m(0) \|_{(H^s)^*} + \int_0^T \|a(s)\|_{(H^s)^*} \, ds ) \,.
\end{equation}
Hence the map $\Psi$ is a contraction for $t$ small enough and it therefore proves the existence and uniqueness of a solution of formula \eqref{IntegralFormulation} for short times.
Remark that Equation \eqref{LipProp} gives an upper bound $t_{lip} := \left(M_1 (\| m(0) \|_{W_2^*} + \int_0^T \|a(s)\|_{W_2^*} \, ds) \right)^{-2}$ such that for every $t<t_{lip}$, $\Psi$ is a contraction. In addition this upper bound is valid at any time $t \in [0,T]$.
Then the existence and uniqueness until time $T$ follows straightforwardly by an iterative application of the short-time result.
\par
The second part of the proof consists in showing existence of solutions for $a \in L^2([0,1],(H^{s-1})^*)$ which is done using a compactness argument. Let $a_n$ converging to $a$ in $L^2([0,1],(H^{s-1})^*)$, then the solution $m_n \in C^0([0,T],(H^{s})^*)$ actually belongs to $H^1([0,1],(H^{s-1})^*)$ because $\| Ad_g^*(m) \|_{(H^{s-1})^*} \leq M_2 \| g \|_{H^s} \| m \|_{(H^{s-1})^*}$. By the Aubin-Lions-Simon theorem, $H^1([0,1],(H^{s-1})^*)$ is compactly embedded in $L^2([0,1],(H^{s})^*)$, thus one can extract a strongly convergent sequence in $L^2([0,1],(H^{s})^*)$. By theorem \ref{Thmdiff_hopf_rinow}, the flow associated with the momentum $m_n$, denoted by  $g_{s,t}^n$ strongly converges in $H^s$ (actually uniformly in $s,t$) to $g_{s,t}$ the flow associated with the limit $m$. Then, it implies that the integrand $Ad_{g^n_{t,s}}^*(a_n(s))$ in Formula \eqref{IntegralFormulaBis} converges to $Ad_{g_{t,s}}^*(a(s))$ in $(H^{s-1})^*$. Since the integrand is bounded uniformly, the Lebesgue  convergence theorem applies and the result is obtained.
\end{proof}

\begin{remark}
\begin{enumerate}
\item The reason why we are not able to treat the case of $a \in L^2([0,1],(H^{s-1})^*)$ is because  the flow map in Theorem  \ref{Thmdiff_hopf_rinow} is only continuous and (possibly) not Lipschitz.
\item Note that Lemma \ref{ThmExistenceEulerian} can be considered as the Eulerian version of Lemma \ref{ThmExistenceLagrangian} and the latter achieves a better result since uniqueness of the solution is proven in $H^{s+1}$. However, our proof of the main theorem will require the use of Lemma \ref{ThmExistenceEulerian} which gives the fact that the Eulerian velocity of the solutions of Lemma \ref{ThmExistenceLagrangian} are bounded in $H^1([0,1],(H^s)^*)$.
\end{enumerate}
\end{remark}

\begin{lemma}\label{ThmWeakComposition}
Let $s > d/2 + 1$,
$\alpha_n \in H^s$ weakly converging to $\alpha$ and $\varphi_n \in \on{Diff}^s$ which strongly converges to $\varphi$, then the composition $\alpha_n \circ \varphi_n$ weakly converges to $\alpha \circ \varphi$.
\end{lemma}

\begin{proof}
We prove the weak convergence by proving that the sequence is bounded and that it weakly converges on a dense set of $H^s$. First, remark that $\| \alpha_n \circ \varphi_n\|_{H^s}$ is bounded in $H^s$ since the composition by a diffeomorphism in $\on{Diff}^s$ is bounded (see \cite[Lemma 2.2]{CompletenessDiffeomorphismGroup}).
\par
Let $m \in (H^{s})^* \cap \mathcal{M}$ where $\mathcal{M}$ denotes the space of Radon measures, consider $\langle \alpha_n \circ \varphi_n , m \rangle_{L^2}$, which can be written by a change of variable as 
\begin{equation}
\langle \alpha_n \circ \varphi_n , m \rangle_{L^2} = \langle \alpha_n , (\varphi_n)_*(m) \rangle_{L^2}\,,
\end{equation}
where $(\varphi_n)_*(m)$ is the pushforward of $m$ by $\varphi_n$. Since $\varphi_n$ is strongly convergent in $\on{Diff}^s$, we have that $(\varphi_n)_*(m)$ strongly converge in $(H^{s})^*$ to $\varphi_*(m)$. Therefore, $\langle \alpha_n , (\varphi_n)_*(m) \rangle_{L^2}$ converges to $\langle \alpha , \varphi_*(m) \rangle_{L^2}$, which gives the result.
\end{proof}

\begin{proof}[Proof of the theorem]
First note that the space $\Om_{0,1}(\mc G_{H^{s'}(\Om,\R^d)})$ is non-empty: Consider a path connecting $\varphi_0$ and  $\varphi_1$ in $\mc G_{H^{s'}(\Om,\R^d)}$, thus, by concatenation of paths, the problem is reduced to $\varphi_0 = \varphi_1 $ where the path can be easily defined on the tangent space in a neighborhood of $\on{Id}$ or in a local chart. 
We will use natural coordinates, i.e. $\on{Id} + H^{s}$ to describe elements of the loop space $\Om_{0,1}(\mc G_{H^{s'}(\Om,\R^d)})$.
The term $\left\| \frac{D}{Dt} \dot{\varphi} \o \varphi\i \right\|^2_{H^{s'}}$ can be written in coordinates $$\| \left( \ddot{\varphi} + \Gamma(\varphi)(\dot{\varphi},\dot{\varphi}) \right) \circ \varphi^{-1} \|_{H^{s'}}^2$$ so that, it is natural to introduce the change of variable $\left( \dot{v} + \Gamma(\varphi)(v,v) \right) \circ \varphi^{-1} := \alpha$ where $\dot{\varphi} = v$.
The variational problem \eqref{SplinesFunctional} can be rewritten as the minimization of 
the functional defined on the Hilbert space $L^2([0,1],H^{s'})$
\begin{equation}
\ell (\alpha) = \int_0^1 \| \alpha(t) \|^2_{H^{s'}}  \ud t \,,
\end{equation}
under the constraint
\begin{equation} \label{ReducedForcedEvolution}
\begin{cases}
\dot{\varphi} = v \\
\dot{v} = -\Gamma(\varphi)(v,v) + \alpha \o \varphi \,,
\end{cases}
\end{equation}
and the boundary conditions, $\varphi(0) = \varphi_0$, $\varphi(1) = \varphi_1$ and $v(0) = v_0$, $v(1) = v_1$.
\par
This functional is lower semi-continuous on $L^2([0,1],H^{s'})$. Let $\alpha_n$ be a minimizing sequence in $L^2([0,1],H^{s'})$ weakly converging to $\alpha$. The condition to be checked is the constraints that have to be satisfied at the limit. Using Lemma \ref{ThmExistenceEulerian}, the sequence $v_n \circ \varphi_n^{-1} \in H^1([0,1],H^{s-1})$ is bounded  and one can extract a strongly converging sequence in $C^0([0,1],H^{s})$ which implies the strong convergence of $\varphi_n(1)$ and $v_n(1)$ in $H^s$. The first consequence is that the boundary constraints $\varphi(1),v(1)$ are satisfied at the limit. It also implies that the term $\Gamma(\varphi_n)(v_n,v_n)$ is strongly convergent in $H^s$ and the term $\alpha_n \circ \varphi_n$ is weakly convergent to $\alpha \circ \varphi$, using Lemma \ref{ThmWeakComposition}. Therefore, we have the equality $v(t) = v(0) + \int_0^t -\Gamma(\varphi(s))(v(s),v(s))  +  \alpha \circ \varphi \ud s$, which implies that the couple $(\varphi,v)$ is the solution of the integral equation associated with System \eqref{ReducedForcedEvolution}.
\end{proof}

\begin{theorem}[Spline interpolation of time sequences]
Let $\varphi_1,\ldots,\varphi_n$ be $n$ diffeomorphisms in $\on{Diff}_0^{s+1}$ and $t_1 <\ldots < t_n$ be a sequence of $n$ positive reals. There exists a path $\varphi(t) \in \on{Diff}_0^{s+1}$, which minimize the acceleration functional
\begin{equation}\label{EqSplinesSequence}
\| \dot{\varphi}(0)\circ \varphi\i \|^2_{H^{s'}} + \int_{t_1}^{t_n} \left\| \frac{D}{Dt} \dot{\varphi} \o \varphi\i \right\|^2_{H^{s'}}  \ud t\,,
\end{equation}
among all curves satisfying $\varphi(t_i) = \varphi_i$ for $i \in 1,\ldots,n$.
\end{theorem}

\begin{proof}
The proof is similar to that of Theorem \ref{ThmSplines} and we do not repeat the arguments here. Note that the penalization on the initial speed seems necessary in order for the curve to stay in a bounded metric ball.\footnote{On the flat 2D torus, straight lines with irrational slopes are dense and they can be parametrized with arbitrarily high velocity so that the infimum of \eqref{EqSplinesSequence} is $0$ without speed penalization.}
\end{proof}

\section{Conclusion}

This theoretical proof of existence was provided to fill in the gap of the variational models proposed in \cite{SinghVN15}. However, we do have treated the case of the induced metric on the space of images, which was also implemented in \cite{SinghVN15}. However, with minor modifications, the approach developed in this article can possibly adapted. On a more theoretical point of view, we leave the open question if the approach can be adapted for $s'>s$.

\bibliographystyle{alpha}
\small{\bibliography{these,SecOrdLand,sum_of_kernels,articles,articles2,Thesis,bibchizat,MesPapiers,references3,preprints,articles4,references,articles5,EntropicNumeric,refs,SecOrdLand,deforNew,shoot}}

\end{document}